\documentclass[12pt]{amsart}

\usepackage{graphicx}
\usepackage{mathrsfs, amsthm, amsmath, amssymb}
\usepackage{appendix} 
\usepackage{bm, hyperref}
\usepackage{cite}
\usepackage{url}
\usepackage{multicol}
\usepackage{enumitem}

\usepackage{fullpage}
\usepackage{amsfonts}
\usepackage{amssymb}

\usepackage{epsfig} 
\usepackage{euscript} 
\usepackage{amsfonts} 
\usepackage{xfrac}

\newcommand{\C}{\mathbb{C}}  
\newcommand{\N}{\mathbb{N}}

\newcommand{\Z}{\mathbb{Z}}

\DeclareMathOperator{\Quot}{Quot}

\newtheorem{main}{Theorem}

\newtheorem{thrm}{Theorem}[section]
\newtheorem{lemma}[thrm]{Lemma}
\newtheorem{prop}[thrm]{Proposition}

\newtheorem{cor}[thrm]{Corollary}
\theoremstyle{plain}
\newtheorem{defn}[thrm]{Definition}

\theoremstyle{remark}

\newcommand{\set}[1]{ \{ #1 \} }
\newcommand{\norm}[1]{\| #1 \|}

\newcommand{\generator}[1]{\langle #1 \rangle}

\title{Classification of Tensor Decompositions of II$_1$ Factors Associated With  Poly-Hyperbolic Groups}

\author[R. de Santiago]{Rolando de Santiago}
\address{Department of Mathematics, University of California Los Angeles, Los Angeles, CA, USA }
\email{rdesantiago@math.ucla.edu}
\thanks{R.dS. was partially supported by  a grant at from Sloan Center for Exemplary Mentoring, NSF Grant DMS \# 1600688  and NSF RTG DMS \# 1344970}

\author[S. Pant]{Sujan Pant}
\address{Department of Science and Mathematics, Alvernia University, Reading, PA, USA}
\email{sujan.pant@alvernia.edu}

\begin{document}
\begin{abstract}
We demonstrate von Neumann algebra arising from an icc group $\Gamma$ in Chifan's, Ioana's, and Kida's class of  poly-$\mathcal{C}_\text{rss} $, such as a poly-hyperbolic group with no amenable factors in its composition series,   satisfies the following rigidity phenomenon discovered in \cite{DHI16} (see also \cite{CdSS17}): every tensor decomposition of the II$_1$ factor $L(\Gamma) $  must arise from direct product decomposition of $\Gamma $ by groups which are poly-$ \mathcal{C}_\text{rss}$.  Through heavy usage and developments of the techniques in \cite{CdSS15}, we improve the second author's and their collaborator's work in \cite{CKP14} by providing group-level criteria for determining whether a group von Neumann algebra is prime: $L(\Gamma) $ is prime precisely  when the group is indecomposable as a direct product of  non-amenable groups. We further demonstrate that all tensor decompositions of finite index subalgebras  of $L(\Gamma) $ correspond to a splitting of $\Gamma $ as a product by groups which are  also  poly-$\mathcal{C}_\text{rss} $ \emph{up to commensurability}.
\end{abstract}
\maketitle

\section{Introduction}

An important topic in  structural properties of II$_1$ factors is the classification of an algebra's tensor decompositions. In his seminal work,  Popa provided the first result in this direction by demonstrating that the non-separable von Neumann algebra arising from the free group on uncountably many generators, $L(\mathbb{F}_S) $, cannot written as a tensor product of II$_1$ factors \cite{Po83}.  Popa called von Neumann algebras which exhibit this indecomposability property  \emph{prime}.  Primeness results for separable II$_1$ factors remained elusive for over a decade; only after Ge's adaptation of Voiculescu's  free probability theory were the separable free group factor(s), $L(\mathbb{F}_n) $ with $ n\geq 2$, shown to be prime  \cite{Ge98}.Ge conjectured that the internal structure of $L(\mathbb{F}_n)$ should in fact have greater restriction on relative commutants of subalgebras.   By developing  $C^*$ algebraic methods, Ozawa provided a sweeping generalization of previous works and confirmed Ge's suspicions by proving that all II$_1$ factors arising  from non-amenable icc hyperbolic groups  are in fact \emph{solid}, i.e.~the relative commutant of any diffuse subalgebra remains amenable.  
Developing the innovative framework of deformation/rigidity,
Popa provided an alternative approach for establishing solidity for the free group factors exclusively contained within von Neumann algebra theory
\cite{Po06}. 
The advent of these methods  paved the way for much of intense research activity over the past decade  \cite{Oz04,OP07, CH08, CI08,CP10,  Fi10,Si10,Va10,CS11,CSU11, SW11, HV12,Bo12, BHR12, DI12, CKP14, Is14, HI15, Ho15, DHI16, Is16}.  

\subsection{Statement of the main results}

In \cite{DHI16},   Drimbe, Hoff and Ioana  prove that if  $\Gamma $is an icc group which is measure equivalent to a direct product of hyperbolic icc groups, then every tensor decomposition of $ L(\Gamma)$ into II$_1$ factors $ P_1\bar\otimes P_2$  must arise from  direct product decomposition of the group $\Gamma$.   In an investigation independent of this project, 
 Chifan, Sucpikarnon, and the first author uncovered  the same phenomenon occurs the setting of von Neumann algebras arising from a large collection of groups stemming from amalgamated free product and HNN extensions \cite{CdSS17}.

A natural extension to the result of \cite{DHI16} is to investigate whether the algebras induced from a subcollection of poly-hyperbolic groups, those whose quotients are non-amenable hyperbolic and more generally 
Chifan, Ioana and Kida's class  $\Quot(\mathcal{C}_\text{rss}) $, admit similar classification of their respective tensor product decompositions. 
Chifan, Kida, and Pant established primeness for the group von Neumann algebras of icc groups in the collection   $ \Quot(\mathcal{C}_\text{rss})\cap NC_1 $  (see \cite[Notation 0.1]{CSU13}).  
Eliminating the assumption that the group lies in the class  $NC_1$, we combine the methods and results  from \cite{PV11,PV12, Va13, CIK13,CKP14, CdSS15} to demonstrate that whenever $\Gamma\in \Quot(\mathcal{C}_\text{rss}) $ is an icc group, every  decomposition of $L(\Gamma) $ as a tensor product of II$_1$ factors is in correspondence with decomposition $\Gamma $ as a direct product.  Moreover, we are also able to show that each factor which decomposes a group $\Gamma $ as a direct product is contained within the collection $\Quot(\mathcal{C}_\text{rss}) $ thereby  providing a group-level criterion for deducing primeness of the resulting factors.

\begin{main}\label{main:Theorem2Tensor}
Let  $\Gamma\in\Quot_n(\mathcal{C}_\text{rss}) $ be an icc group and suppose there exist  II$_1$ factors $P_1, P_2$ such that $L(\Gamma)\cong P_1\bar\otimes  P_2 $.  Then there exist subgroups $  \Gamma_1, \Gamma_2\leqslant\Gamma$, a unitary $u\in \mathcal{U}(L(\Gamma)) $, positive integers  $n_1,n_2 $, and  a scalar $t> 0 $  such that:
\begin{enumerate}
\item $\Gamma=\Gamma_1\times \Gamma_2 $ with $\Gamma_i\in \Quot_{n_i}(\mathcal{C}_\text{rss}) $ and $n_1+n_2=n $; and
\item $P_{1}^t=u L(\Gamma_1)u^* $ and $P_{2}^{1/t}=u L(\Gamma_2)u^* $.
\end{enumerate}
\end{main}
Recall that A.~Connes' landmark result implies that no algebraic information of an icc amenable group  $\Gamma$ be recovered from the algebra $L(\Gamma) $\cite{Co76}. For instance,  $L(S_\infty)\cong L(S_\infty\times S_\infty) $, but $S_\infty $ does not admit a direct product decomposition.  Despite the fact the assignment $\Gamma\mapsto L(\Gamma) $ often fails to retain properties of the underlying group, Theorem \ref{main:Theorem2Tensor} demonstrates that   the algebra  $L(\Gamma) $ can indeed ``remember'' the direct product structure of the group $\Gamma $ for a prominent class of groups.

Since the  class $\mathcal{C}_\text{rss} $ is closed under \emph{commensurability}, so is each class $\Quot_n(\mathcal{C}_\text{rss}) $.   Notably, this property allows us to  classify not only every tensor product decomposition of $L(\Gamma) $ but also classify   every tensor product decomposition of finite index subalgebras of $L(\Gamma) $.  

\begin{main}
Let $\Gamma\in \Quot_n(\mathcal{C}_{\text{rss}}) $ be an icc group and $p\in \mathcal{P}(L(\Gamma)) $ be a projection.  Suppose there exists $P_1,\ldots, P_k\subset pL(\Gamma )p$ commuting diffuse subalgebras  generating a finite index subalgebra of $pL(\Gamma) p$.  Then  there exists a projection  $p_i\in P_i $, finite index subfactors $ Q_i\subset p_iP_ip_i$, groups $\Gamma_1,\ldots, \Gamma_k $, and a unitary $u\in pL(\Gamma)p $ such that
\begin{enumerate}
\item $Q_i\subset p_iu^*L(\Gamma_i)up_i $ is a finite index inclusion of algebras so that $\Gamma_i\in \Quot_{n_i}(\mathcal{C}_\text{rss}) $ and $\sum_{i=1}^k n_i=n $, and
\item $\Gamma $ is commensurable to $\Gamma_1\times \cdots\times \Gamma_k$.
\end{enumerate}
In particular, $L(\Gamma) $ is prime if and only if $\Gamma $ is not commensurable to a non-trivial direct product of groups in $\Quot (\mathcal{C}_\text{rss}) $.  

\end{main}
As a simple corollary to the theorems above, we obtain a unique prime decomposition similar to that of Ozawa and Popa for these group von Neumann  algebras \cite[Corollary 3]{OP03}.  

\begin{main}
Suppose $ \Gamma_1,\ldots, \Gamma_m $ and $\Lambda_1,\ldots, \Lambda_m $ are groups such that $\Gamma_i\in \Quot_{n_i}(\mathcal{C}_\text{rss}) $ and $\Lambda_{i}\in \Quot_{m_j}(\mathcal{C}_\text{rss}) $.  If $L(\Lambda_i) $ and $L(\Gamma_j) $ are prime II$_1$ factors so that $L(\Gamma_1\times \cdots\times \Gamma_m)\cong L(\Lambda_1\times \cdots\times \Lambda_n) $, then  $n=m $ and we have $L(\Gamma_i)\cong L(\Lambda_i)$, up to		 permutation and amplification.
\end{main}

%
%
%

%
%
%
\subsection{Comments on the Proof}
The class of groups $\mathcal{C}_\text{rss} $ is comprised of all groups which satisfy the  dichotomy result for the classification of normalizes for crossed products by non-elementary hyperbolic groups  of Popa and Vaes \cite{PV12}.  We utilize this criterion to infer algebraic and analytic structural properties $\Gamma $ and $L(\Gamma) $, respectively.  
 To give a brief outline of the strategy involved, assume $\Gamma\in \Quot_n(\mathcal{C}_\text{rss}) $ is an $n$-step extension by groups in $\mathcal{C}_\text{rss}$.  

We begin by exploiting the aforementioned analytic aspects of $L(\Gamma) $  to conclude that whenever   $\Gamma=\Gamma_1\times \Gamma_2 $,  then either: $\Gamma_1 $ is finite and  $\Gamma_2\in \Quot_n(\mathcal{C}_\text{rss}) $; $\Gamma_2 $ is finite and $\Gamma_1\in \Quot_{n}(\mathcal{C}_\text{rss}) $; or there exist positive integers $n_1, n_2 $ so that $\Gamma_i\in \Quot_{n_i}(\mathcal{C}_\text{rss}) $.  This result holds in greater generality. If we instead assume $ \Gamma$ has two commuting subgroups which generate a finite index subgroup of $\Gamma $, than a parallel statement holds.

In \cite{CdSS15}, Chifan, Sinclair and the first author prove that whenever two groups $\Gamma $ and $\Lambda$ are such that $\Gamma=\Gamma_1\times \cdots\times\Gamma_k $ is a $k$-fold product of icc hyperbolic groups such with $L(\Gamma)\cong L(\Lambda) $, then $\Lambda $ is necessarily a $k$-fold product of non-amenable groups.  This so-called  product rigidity result   describes a non-trivial procedure whereby one transfers the existence of commuting algebras into the existence of large commuting subgroups of the mystery group $\Lambda $, which  can then be perturbed to decompose the group  $\Lambda$ into a non-trivial direct product.  Most notably, this provides a case where the analytic properties of the algebra $L(\Lambda) $ can be correlated to algebraic data of the group $\Lambda $.  Adapting this procedure to the situation where $\Gamma\in \Quot_n(\mathcal{C}_\text{rss}) $ and $L(\Gamma)\cong P_1\bar\otimes P_2 $ for some pair  $P_i$ of II$_1$ factors, then the classification for normalizers allows us to assume that a corner of $P_1$ can be embedded, in the sense of Popa, into $L(\Gamma_1) $ for some  normal subgroup  $\Gamma_1\triangleleft\Gamma $.  $\Gamma_{1} $ is necessarily  an $n-1$-step extension by groups in $\mathcal{C}_\text{rss} $.  This, in conjunction with intertwining results found  within \cite{CKP14}, will lead to the following two cases:
\begin{enumerate}
\item A corner of $P_1 $ is conjugate to a finite index subalgebra of a corner of $L(\Lambda) $, or
\item A corner of $P_1 $ and its relative commutant generate a finite index subalgebra of a corner of $L(\Lambda) $.
\end{enumerate}
The former case allows us to identify $P_1 $ with $L(\Gamma_1) $ and applying a procedure presented in \cite{CdSS15} we have that $\Gamma_1 $ the centralizer $\Gamma_2=C_\Gamma(\Gamma_1) $ effectively decompose $\Gamma $ as a direct product.  While case (2) requires a more detailed analysis, the proof essentially follows the same line of argument as case (1).

\subsection*{Acknowledgments}
The authors would like to thank Professor Ionu\c{t} Chifan for proposing the original idea from which this work was developed.  We are deeply indebted to him for his persistence, patience, guidance,  mentorship, and encouragement throughout the development of this work.

\section{Preliminaries}

A von Neumann  algebra $M$ is a unital, SOT closed, $*$-subalgebra of $B(H) $ for some Hilbert space $H$.   As proven by Murray and von Neumann ,  $M\subset B(H)$ is a von Neumann algebra if and only if $ M=M''$, where $M'=\set{y\in B(H): xy=xy \,\,\, \forall x\in M} $ is the commutant of $M$ \cite{MvN36}.  A von Neumann algebra $M$ is  termed a \emph{factor} if the center is trivial, i.e.~$ \mathcal{Z}(M)=M\cap M'\cong \C $. 

Unless explicitly stated, we shall assume  all von Neumann algebras considered are separable (in the sense that they act upon a countable Hilbert space) and all inclusions of von Neumann algebras $P\subseteq M $ are considered unital.   
 Whenever $P\subseteq M $ is an inclusion of algebras, $P'\cap M $ is the \emph{relative commutant of $P$ inside $M$}.  The \emph{center} of a von Neumann algebra is $M\cap M' $, denoted simply as $\mathcal{Z}(M) $.  Given von Neumann subalgebras $P,Q\subset M $, $P\vee Q $ denotes the smallest subalgebra of $M$ containing both $P$ and $Q$, and if $\set{P_i\subseteq M: i\in I } $ is a collection of von Neumann sub algebras of $M$, 
 $\bigvee_{i\in I} P_i  $ denotes
 the smallest subalgebra of $M$ containing $\bigcup_{i\in I}P_i $.  
 If $P\subseteq M $  is a von Neumann subalgebra, the \emph{normalizer} of $P$ inside $M$ is the subgroup of unitaries $\mathcal{N}_M(P)=\set{u\in \mathcal{U}(M) :  uPu^*=P} $.  A von Neumann algebra is \emph{tracial} if there exists a normal, faithful, tracial state $\tau:M\to \C $ and \emph{finite} if $\tau(1_M)=1 $.  When $M$ is a factor, any tracial state $\tau $ will be the unique such state. 

A von Neumann algebra $M$ is type II$_1 $ if $M$ is infinite dimensional and is tracial.  Whenever $M$  is a type II$_1$ factor,  $\tau(\mathcal{P}(M))=[0,1] $. For every scalar $0\leq t\leq 1 $, the \emph{amplification of $M$ by $t$} is the von Neumann algebra $M^t=pMp $, where $p\in \mathcal{P}(M)$ is any projection with $\tau(p)=t$.  It is well-known that the isomorphism class of $ M^t$ is independent of the choice of projection $p$ since for any projection $ q$ such that $\tau(p)=\tau(q)=t $, the algebras $pMp $ and $qMq $ are unitarily conjugate via a unitary in $M$.  We may extend the definition $M^t $ to all scalars  $t>0 $ by taking a sufficiently large integer $n$ and  projection $p$ in $M_n(M) $ with $\tau\circ\operatorname{Tr}_n(p)=t $, where $\operatorname{Tr}_n:M_n(\C)\to \C $ is the standard trace on $n\times n$ matrices.

To a  discrete countable group $\Gamma $, Murray and von Neumann describe how one associates  a von Neumann algebra via the \emph{left regular representation}: 
Consider the Hilbert space $\ell^2(\Gamma) $, the space of square summable sequences indexed by $\Gamma $.  For every $\gamma\in \Gamma$, we can define a unitary  $u_\gamma\in B(\ell^2 (\Gamma)) $ by linearly extending the map
$u_\gamma(\delta_g)=\delta_{\gamma g} $
where the collection $\set{\delta_g}_{g\in \Gamma}$ are the standard basis vectors for $\ell^2(\Gamma) $. The group von Neumann algebra $L(\Gamma)\subset B(\ell^2 (\Gamma)) $ is the von Neumann algebra generated by the  \emph{canonical unitaries}  $\set{u_\gamma}_{\gamma\in \Gamma} $. All groups considered herein will be discrete, unless specified otherwise.

Fix a group discrete  $\Gamma $.  We denote conjugation of $\gamma\in \Gamma $ by $\lambda\in \Gamma $ as $\gamma^\lambda=\lambda^{-1}\gamma \lambda$.  For each
  group element $ \gamma\in \Gamma$, and a subset $S\subset \Gamma $ we denote the orbit of $\gamma $ under conjugation by $S$ as $ \gamma^S=\set{\gamma^s: s\in S}$. A group is said to be of \emph{infinite conjugacy class}, hereby initialized as icc, if for every  $\gamma\in \Gamma\setminus \set{e} $ we have $|\gamma^\Gamma|=\infty $.  A well-know result  of Murray and  von Neumann states the algebra  $L(\Gamma) $ is a type II$_1$ factor  if and only if $\Gamma $ is an icc group \cite{MvN43}.  In this case, the  the unique normal faithful tracial state  $\tau : L(\Gamma)\to \C $ given by $\tau(x)=\generator{x\delta_e, \delta_e} $. We let
 $C_\Gamma(S) =\set{\gamma\in \Gamma: \gamma s =s\gamma\; \forall \, s\in S}$ denote the centralizer of a set $S\subset \Gamma $, and  $ \mathcal{V}_\Gamma(S)=\set{\gamma\in \Gamma :  |\lambda^S|<\infty}$ will denote the \emph{virtual centralizer of $S $ inside $\Gamma $}.  Given subsets $S,T\subset \Gamma $, $ST=\set{st  :  s\in S,\, t\in T} $ is the collection of all products of elements in $S $ and $T$.

\subsection{Inclusions of Algebras}
A von Neumann algebra $M$ is \emph{diffuse} if there does not exist a non-zero projection $p\in \mathcal{P}(M) $ so that $pMp=\C p $, i.e.~$M$ does not contain any minimal projections.  $M$ is \emph{prime} if $M$ cannot be written as a tensor product of any two diffuse von Neumann algebras.

Given a tracial von Neumann algebra $(M,\tau)$, we define $L^2(M) $ as the Hilbert space completion of $M$ with regards to the sesquilinear form $\tau:M\times M\to \C $ given by $\generator{x,y}=\tau(y^*x) $.   
Let $P\subset M$ be an inclusion tracial of von Neumann algebras with trace $\tau $. The {\it basic construction} \cite{Ch79} $\langle M,e_P\rangle\subset \mathbb B(L^2(M))$ is the smallest von Neumann algebra generated by $M$ and the orthogonal projection $e_P:L^2(M)\rightarrow L^2(P)$. $ \langle M,e_P\rangle$ is endowed with a faithful \emph{semi-finite} trace $Tr$ given by $Tr(xe_Py)=\tau(xy)$, for all $x,y\in M$. Also, we note that $E_P:={e_P}_{|M}:M\rightarrow P$ is the unique $\tau$-preserving conditional expectation onto $P$.



 If $P\subset M$ is an inclusion of von Neumann algebras, a state $\phi:M\to \C $ is said to be $P$-central if $\phi(mx)=\phi(xm) $ for every $x\in P $ and every $m\in M $.  A tracial von Newmann algebra $(M,\tau) $ is  \emph{amenable} if there exists an $M$-central state $\phi:B(L^2(M))\to \C $ so that $\phi|_M=\tau$.  By the celebrated result of A.~Connes, a von Neumann algebra is amenable if and only if it is approximately finite dimensional, i.e.~$M=\varinjlim M_n $ for an increasing sequence of finite dimensional algebras $M_n$\cite{Co76}.     
 In the spirit of  Popa's  intertwining techniques (see \ref{corner}), Ozawa and Popa described a relative version of this concept which proves useful in the analysis of structural properties of subalgebras.:

\begin{defn} \cite[Section 2.2]{OP07}
Let $(M,\tau)$ be a tracial von Neumann algebra, $p\in M$ a projection, and $P\subset pMp,Q\subset M$ von Neumann subalgebras. We say that $P$ is {\it amenable relative to $Q$ inside $M$} if there exists a $P$-central state $\phi:p\langle M,e_Q\rangle p\rightarrow\mathbb C$ such that $\phi(x)=\tau(x)$, for all $x\in pMp$.
\end{defn}
The classical notion of amenability is recovered by letting $ Q=\mathbb{C}$. More generally, if $Q$ is amenable and $P$ is amenable relative to $Q$, then $P$ is necessarily amenable. 
This definition is fruitful for numerous reasons, not the least of which its strong semblance to that definition of relative amenability for groups: 
as an amenable group Neumann algebra corresponds to an amenable group, so does relative amenability of  group von Neumann algebras coincide with relative amenability of the groups, i.e.~given a group $\Gamma $ and a pair of subgroups  $\Lambda_1,\Lambda_2\leqslant \Gamma $,  then  $\Lambda_1<\Gamma $ is amenable relative to $\Lambda_2 $ inside $\Gamma$ is amenable relative to $\Lambda_2 $ inside $\Gamma $  if and only if $L(\Lambda_1 )$ is amenable relative to $L(\Lambda_2)$ inside $L(\Gamma) $.
Should the ambient algebra $M$  be understood from context, we simply state that $P $ is amenable relative to $Q$.

Given an inclusion $P\subset M $ of II$_1$ factors, the \emph{Jones index} of $P\subset M$, denoted $[M:P]$, is the dimension of $L^2(M) $ as a (left) $P $ module.  For an arbitrary inclusion of tracial von Neumann algebras $P\subset M$, Pimnser and Popa defined a probabilistic notion of index via the conditional expectation from $M$ onto $P$.
\begin{defn}[Pimsner $\&$ Popa, \cite{PP86}] 
Let $(M,\tau)$ be a tracial von Neumann algebra  with a von Neumann subalgebra $P$. Let $$\lambda=\inf\;\{\|E_P(x)\|_2^2/\|x\|_2^2\;:\; x\in M_{+}\}.$$
The {\it Pimnser-Popa index of the inclusion $P\subseteq M$} is defined as $[M:P]_{\text{PP}}=\lambda^{-1}$, under the convention that $\frac{1}{0}=\infty$. 
\end{defn}

\begin{thrm}[\cite{Jo81, PP86}]\label{thrm:IndexofvnAlgebras}
Suppose  $P\subset M$ is an inclusion of tracial von Neumann algebras.  Then the following hold:
\begin{enumerate}
\item If $P\subset M $ is an inclusion of II$_1$ factors, then $[M:P]_{PP}=[M:P] $
\item If $[M:P]_{PP} <\infty $ and $p\in P $ is a projection, $[pMp:pPp]<\infty $;
\item If $P $ is a II$_1$ factor and $[M:P]_{PP}<\infty $ then $P'\cap M $ is finite dimensional;
\item If $P\subset M $ is an inclusion of II$_1$ factors with $[M:P]_{PP}<\infty $, then $\dim_\C (P'\cap M)<\infty$.  
\end{enumerate}
\end{thrm}

Whenever $\Gamma $ a discrete group and $\Lambda$ a finite index subgroup, it follows  $[L(\Gamma):L(\Lambda)]_{PP} $ is finite as well.  In \cite{CdSS15}, the authors provide a generalized converse to this fact.
\begin{prop}[\cite{CdSS15}]
Let $\Omega\leqslant\Lambda\leqslant \Theta $ be groups such that there exists projections $p\in L(\Omega) $ and $z\in L(\Lambda) '\cap L(\Theta) $ so that $pz\neq 0 $ and $[pL(\Lambda) pz: pL(\Omega) pz]_{\text PP}<\infty $.  Then $[\Lambda:\Omega]<\infty $.
\end{prop}

\section{Popa's intertwining techniques} 

To describe the structure of von Neumann algebras, Popa introduced a powerful new conceptual framework: \emph{deformation/rigidity theory}. This methodology includes a powerful criteria for identifying intertwiners between subalgebras of type II$_1$ factors, now called \emph{Popa's intertwining-by-bimodules techniques}. Much of the recent progress in classifying von Neumann  algebras can be  largely attributed to this philosophy.

\begin {thrm}[Popa, \cite{Po03}]\label{corner} Let $(M,\tau)$ be a separable tracial von Neumann algebra and $P,Q$ be two (not necessarily unital) von Neumann subalgebras of $M$. 
The following are equivalent:
\begin{enumerate}

\item There exist  non-zero projections $p\in P, q\in Q$, a $*$-homomorphism $\theta:pPp\rightarrow qQq$  and a non-zero partial isometry $v\in qMp$ such that $\theta(x)v=vx$, for all $x\in pPp$.

\item Let $ \mathcal{G}\subset P$ be a group of unitaries generating $P $. There is no sequence $u_n\in \mathcal{G}$ satisfying $\|E_Q(xu_ny)\|_2\rightarrow 0$, for all $x,y\in M$.
\end{enumerate}
\end{thrm}
If either of the equivalent conditions hold, we say \emph{$P$ intertwines into $Q$ over $M$}, denoted $P\prec_M Q $.  Should $Pz\prec_M Q$ for every $z
\in P'\cap M $, then we say $P$ \emph{strongly intertwines into $Q$} and denote this property by $P\prec_M^s Q $.  Whenever the ambient algebra $M$ is clear from context, we suppress the subscript and write $P\prec Q $ (or $P\prec^s Q$) whenever $P$ (strongly) intertwines into $Q$.

Fix a trace preserving action $\Gamma\curvearrowright (N,\tau) $ of a group $\Gamma $ on a tracial von Neumann algebra $(N,\tau) $  and let $S $ be a collection of subgroups  of $\Gamma $.  A set  $\mathcal{F}\subset \Gamma $ is \emph{small relative to} $\mathcal{S}$ if $\mathcal{F} $  is contained in  a finite union of left/right translates of groups in $\mathcal{S} $, i.e.~there exist a finite collection  $g_1,\ldots, g_j, h_1,\ldots h_j\in \Gamma $, $ \Sigma_i\in S$ so that $\mathcal{F}\subset \bigcup_{i=1}^jg_i\Sigma _i $ \cite{BO08}.  Clearly, if $ S$ is a collection of normal subgroups of $\Gamma $, then we need only consider right translates of $S$.  
A resource detailing  the interplay between small sets and  intertwining techniques can be found in \cite{Va10}.

In general, if $P_1, P_2\subset M $ are subalgebras such that $P_1 $ and $P_2 $ both intertwine into a common algebra $Q$, one cannot make any conclusion about $P_1\cap P_2 $ or $ P_1\vee P_2$ without  additional assumptions.  We proceed by describing an instance where if one has multiple commuting subalgebras $P_1,\ldots, P_k $ of $M$ each of which strongly intertwine into  subalgebras $Q_1,\ldots, Q_k $ respectively, one can construct a common subalgebra $Q\subset M$ where $\bigvee_{i=1}^k{P_i} \prec Q $.
\begin{prop}\label{prop:StrongIntertwiningMultiple}
Let $\Gamma\curvearrowright (N ,\tau) $ be a trace preserving action of a group on a II$_1$ factor $N$ and denote by $M=N \rtimes \Gamma $.  

 Suppose there exists normal  subgroups $\Sigma_1,\ldots, \Sigma_k\triangleleft\Gamma $ and pairwise commuting subalgebras $P_1,\ldots, P_k\subset M $ so that $P_i\prec_M^s N\rtimes \Sigma_i  $. Then if $ \Sigma=\Sigma_1\cdots \Sigma_k$,
\begin{align*}
\bigvee_{i=1}^kA_i \prec_M N\rtimes \Sigma.
\end{align*}
\end{prop}

\begin{proof}
 Given a set $\mathcal{F}\subset \Gamma$, $P_\mathcal{F} $ is the orthogonal projection from $L^2(M) $ to the closed linear span of $\set{n_gu_g : n_g\in n, g\in \mathcal{F}} $.
Fixing $1>\varepsilon>0 $,   $P_1\prec_M^s N\rtimes \Sigma_1 $ implies  there exists  a set $\mathcal{F}_1 \subset \bigcup_{i=1}^j \Sigma_1 \gamma_i $ small relative to $\Sigma_1 $ so that
$$\norm{a_1-\mathcal{P}_{\mathcal{F}_1}(a_1)}<\varepsilon/2. $$ 
Recursively choose $\mathcal{F}_i\subset \cup_{j=1}^{j_i} \Sigma_i\gamma_{j,i} $ small relative to $\Sigma_i $ so that whenever $a_i\in A_i $ with $\norm{a_i}\leq 1 $ we have
\begin{align}\label{eq:ProjectionEstimate}
\norm{a_i- \mathcal{P}_{\mathcal{F}_i} (a_i)}\leq \varepsilon/(kj_1\cdots j_{i-1}).
\end{align}
Now since the algebras $P_1,\ldots, P_k $ pairwise commute, we see
\begin{align}\label{eq:ProjectionCommutEstiname}
\norm{a_1\cdots a_k- \mathcal{P}_\mathcal{F}(a_1\cdots a_k)}_2\leq& \norm{a_1\cdots a_k -\mathcal{P}_{\mathcal{F}_1}(a_1)\cdots \mathcal{P}_{\mathcal{F}_k}(a_k)}_2
\end{align}
Combining equations \ref{eq:ProjectionEstimate} and \eqref{eq:ProjectionCommutEstiname}, we now have
\begin{align*}
\norm{a_1\cdots a_k- \mathcal{P}_\mathcal{F}(a_1\cdots a_k)}_2\leq&
\norm{(a_1-\mathcal{P}_{\mathcal{F}_1}(a_1))(a_2\cdots a_k)}_2 \\\nonumber
&+\norm{\mathcal{P}_{\mathcal{F}_1}(a_1)(a_2\cdots a_k -\mathcal{P}_{\mathcal{F}_2}(a_1)\cdots \mathcal{P}_{\mathcal{F}_{k}  }(a_k)  )}_2\\
\leq & \varepsilon/k +\norm{\mathcal{P}_{\mathcal{F}_1}(a_1)(a_2\cdots a_k -\mathcal{P}_{\mathcal{F}_2}(a_1)\cdots \mathcal{P}_{\mathcal{F}_{k}  }(a_k)  )}_2.
\end{align*}
Since $\mathcal{F}_1 $ is contained $j_1$ left translates of $\Sigma_1 $, $\norm{\mathcal{P}_{\mathcal{F}_1}(a_1)}_2\leq j_1\norm{a_1}_2 $.  Thus the  inequality above implies
\begin{align*}
\norm{a_1\cdots a_k- \mathcal{P}_\mathcal{F}(a_1\cdots a_k)}_2\leq & \varepsilon/k + j_1\norm{a_1}\norm{a_2\cdots a_k -\mathcal{P}_{\mathcal{F}_2}(a_1)\cdots \mathcal{P}_{\mathcal{F}_{k}  }(a_k)  }_2\\
\leq & \varepsilon/k + j_1\norm{(a_2-\mathcal{P}_{\mathcal{F}_2}(a_2))(a_3\cdots a_k)}_2  \\
&+j_1 \norm{\mathcal{P}_{\mathcal{F}_2}(a_2)(a_3\cdots a_k-\mathcal{P}_{\mathcal{F}_3}(a_3)\cdots\mathcal{P}_{\mathcal{F}_k}(a_k))}_2 \\
\leq & \varepsilon/k+ j_1\norm{a_2-\mathcal{P}_{\mathcal{F}_2}(a_2)}_2\\
 &j_1 \norm{\mathcal{P}_{\mathcal{F}_2}(a_2)(a_3\cdots a_k-\mathcal{P}_{\mathcal{F}_3}(a_3)\cdots\mathcal{P}_{\mathcal{F}_k}(a_k))}_2 \\
 \leq & 2\varepsilon/k +j_1j_2\norm{a_3\cdots a_k -\mathcal{P}_{\mathcal{F}_3}(a_3)\cdots \mathcal{P}_{\mathcal{F}_k}(a_k)}_2.
\end{align*}
Repeated application of this procedure will then yield
\begin{align}
\norm{a_1\cdots a_k- \mathcal{P}_\mathcal{F}(a_1\cdots a_k)}_2<\varepsilon.\label{eq:NormPositive}
\end{align}
Note that the normality of the groups $\Sigma_i $ yield that the set $\mathcal{F}=\mathcal{F}_1\cdots\mathcal{F}_k $ is contained in a finite union of right translates of $\Sigma=\Sigma_1\cdots \Sigma_k\triangleleft \Gamma$. Thus there exist $\lambda_1,\cdots,\lambda_n\in \Gamma $ so that
\begin{align*}
\mathcal{F}\subset \bigcup_{i=1}^n \Sigma\lambda_i.
\end{align*} 
Thus $\norm{\sum_{i=1}^nP_{\mathcal{F}} ( a_1\cdots a_ku_{\lambda_{i}} )  }^2_2\geq \norm{\mathcal{P}_\mathcal{F}(a_1\cdots a_k)}_2^2>1-\varepsilon$ for every $a_-\in P_i$ with $\norm{a_i}\leq 1 $.  
\begin{align*}
\sum_{i=1}^n \norm{E_{N\rtimes \Sigma}(w u_{\lambda_i})}_2^2
   \geq \norm{\sum_{i=1}^n E_{N\rtimes \Sigma}(wu_{\lambda_i})}_2^2>1-\varepsilon
\end{align*} 
where $w $ is any unitary of the form $w=u_1\cdots u_k $ with $u_i\in \mathcal{U}(P_i) $.  As unitaries of this form generate $ P_1\vee\cdots \vee P_k$,  Theorem \ref{corner} establishes the result.
\end{proof}


\section{Finite Step Extensions of Groups}
Fixing a class of groups $\mathcal{C} $, a discrete countable group $\Gamma $ is a \emph{finite step extension by $\mathcal{C}$} if there exists a
exists  a chain of groups and homomorphisms
\begin{align}\label{eq:QuotnC}
\Gamma=\Gamma_n \overset{\pi_n}\rightarrow \Gamma_{n-1} \overset{\pi_{n-1}}\rightarrow\cdots \overset{\pi_2}\rightarrow\Gamma_1\overset{\pi_1}\rightarrow 1
\end{align}
such that for every $1\leq k\leq n $, $\pi_k:\Gamma_k\to \Gamma_{k-1} $ is a  surjective homomorphism with $\ker(\pi_k)\in \mathcal{C}$.
For each $ n\in \N$, $\Quot_n(\mathcal{C}) $ denotes the collection of all groups which are  $n$-step extensions by $\mathcal{C} $. We denote the collection of all finite step extension by $\mathcal{C} $ by  $ \Quot(\mathcal{C})=\cup \Quot_n(\mathcal{C}) $.  
Combining the results in both \cite{CIK13} and \cite{CKP14},  we have the following facts for the collection of groups $\Quot(\mathcal{C}) $.

\begin{prop}[\cite{CIK13,CKP14}]\label{prop:QuotProperties}
Let $\mathcal{C} $ be any class of groups.  
\begin{enumerate}
\item If $\rho: \Lambda\to \Gamma $ with $\Gamma\in \Quot_n(\mathcal{C}) $ and $ \ker(\rho)\in \mathcal{C}$, then $\Lambda\in \Quot_{n+1} (\mathcal{C})$.
\item If $\Gamma_i\in \Quot_{n_i}(\mathcal{C}) $,  $\Gamma_1\times \cdots\times \Gamma_k\in \Quot_{n_1+\cdots+n_k} (\mathcal{C})$.
\item If $\mathcal{C} $ is closed under commensurability (up to finite kernel), then so is $\Quot_n(\mathcal{C}) $.
\item If $\Gamma\in \Quot_n(\mathcal{C}) $ with $\pi_n:\Gamma_k\to \Gamma_{k-1} $ a family as in the definition of $\Quot_n(\mathcal{C}) $ and $\rho_k:=\pi_2\circ\cdots \circ\pi_n $, then $\ker \rho_n\in \Quot_{n-1}(C) $.
\end{enumerate}
\end{prop}
A generalization of part (4) of Proposition \ref{prop:QuotProperties} may be verified using the following construction.    Given any group $\Gamma\in \Quot_n(\mathcal{C})$ with $ n>2 $ we may recursively define the following family of groups:  First let $\Gamma_n^{(0)}=\Gamma_n $. For $0<j\leq n-1 $, suppose we have
 $$\Gamma^{(j-1)}_n\overset{\pi_n}\to \Gamma_{n-1}^{(j-1)}\overset{\pi_{n-1}}\to\cdots \overset{\pi_{j+1}}\to \Gamma_{j}^{(j-1)}\to 1 $$
with be a sequence of surjections with $\ker(\pi_i)\in \mathcal{C} $. Defining $\rho_k^{(j-1)}=\pi_{j+1}\circ \cdots\circ \pi_k $ and $\Gamma^{(j)}_k=\ker\rho^{(j-1)}_k $, by appropriately restricting $\pi_k $ we now have
 $$\Gamma^{(j)}_n\overset{\pi_n}\to \Gamma_{n-1}^{(j)}\overset{\pi_{n-1}}\to\cdots \overset{\pi_{j+1}}\to \Gamma_{j}^{(j)}= 1 $$
is a chain satisfying the conditions implying $\Gamma_n^{(j)}\in \Quot_{n-j}(\mathcal{C}) $.  More generally we have for $0\leq j< k\leq n $:
\begin{itemize}
\item $\Gamma_k^{(j)}\in \Quot_{k-j}(\mathcal{C}) $,
\item $\Gamma_k^{(j)}\triangleright \Gamma_k^{(j-1)} $,
\item $\Gamma_k^{(j)}/ \Gamma_k^{(j-1)}\in \mathcal{C} $,
\item $\Gamma_n\triangleright \Gamma_n^{(1)}\triangleright \cdots \triangleright \Gamma_n^{(n-1)}\triangleright 1 $ with $\Gamma_{n-1}^{(j)}/ \Gamma_{n}^{(j+1)}\in \mathcal{C} $
\end{itemize}
Hence an equivalent characterization of $\Gamma\in \Quot_n(\mathcal{C})  $ is  $\Gamma $ is \emph{poly-$\mathcal{C} $} with Hirsch length $n$.  

If $\mathcal{C} $ is a class of groups closed under commensurability (up to finite kernel), then we generalize the definition of $ \Quot_n( \mathcal{C}) $ there exists a chain as in \eqref{eq:QuotnC}
\begin{equation}
\Gamma_n\to \Gamma_{n-1}\to\cdots \to \Gamma_1\to 1
\end{equation}
 with $\Gamma $ commensurable (up to finite kernel) to $\Gamma_n $.  In this situation, we must take care as commensurabilty may introduce unexpected  variability. For instance, if we take  the family of all non-amenable free groups $\mathcal{F}$, naturally $\mathbb{F}_4\in\Quot_1(\mathcal{F}) $. The cannonical surjection $\mathbb{F}_4\to \mathbb{F}_2 $ demonstrates the fact $\mathbb{F}_4\in \Quot_2(\mathcal{F}) $. In general, $\mathbb{F}_{2n}\in \Quot_n(\mathcal{F}) $.  As all non-amenable free groups are commensurable, $\mathbb{F}_2\in \Quot_n(\mathcal{F}) $ for every $n $. Thus we impose the following minimality condition in the definition of $\Quot_n(\mathcal{C}) $: 
\begin{defn}\label{defn:QuotnCCommensurable}
Let $\mathcal{C} $ be a class of groups closed under commensurabilty (up to finite kernel).  $\Gamma\in\Quot(\mathcal{C}) $ if there exists a chain of surjections
\begin{align*}
\Gamma_n \overset{\pi_n}\rightarrow \Gamma_{n-1} \overset{\pi_{n-1}}\rightarrow\cdots \overset{\pi_2}\rightarrow\Gamma_1\overset{\pi_1}\rightarrow 1
\end{align*}
so that $\Gamma $ is commensurable to $\Gamma_n $, $\ker(\pi_k)\in \mathcal{C} $.  
$\Gamma\in \Quot_n(\mathcal{C}) $ if $n $ is the smallest positive integer such that $\Gamma $ is a $k $-step extension by $\mathcal{C} $. 

\end{defn}

\section{Relative Strong Solidity and Poly-Hyperbolic Groups}\label{sec:RSSGRoups}

A von Neumann algebra $M$ is said to be \emph{solid} if for any diffuse subalgebra $D\subset M$, the relative commutant $D'\cap M $ is amenable; $M$ is \emph{strongly solid} if the normalizer $\mathcal{N}_M(A) $ of any diffuse amenable subalgebra $A\subseteq M $ generates an amenable subalgebra of $M$.  When $M$ is non-amenable, strong solidity implies solidity, which in turn yields primeness.  

\begin{defn}[\cite{CKP14}]
Given a group $\Sigma $, $\mathcal{C}_{\text{rss}} (\Sigma)$ is the collection of all non-amenable, exact groups  containing $\Sigma $ as a proper malnormal subgroup so that the following holds:\\
Assume $\Gamma\curvearrowright N $ is any trace preserving action on a tracial von Neumann algebra $N$ and denote $M=N\rtimes \Gamma$. Let $p \in M$ be a projection and $P\subset pMp $ a subalgebra amenable relative to $N\rtimes \Sigma $.  Then either
\begin{enumerate}
\item $P\prec_{M} B\rtimes \Sigma $, or
\item $\mathcal{N}_{pMp}(P)'' $ is amenable relative to $N\rtimes \Sigma$. 
\end{enumerate}	 
\end{defn}
 Our work  concerns the case where $\Sigma=\set{e} $ and thus  for ease of notation we set $\mathcal{C} _{\text{rss}}:=\mathcal{C}_{\text{rss}}(\set{e})$.  The following is a summary of known facts of the class $\mathcal{C}_\text{rss} $  
\begin{enumerate}
\item \cite{PV11} Any weakly amenable group with positive first $\ell^2 $-Betti number is in $\mathcal{C}_{\text{rss}} $.
\item\label{ex:HyperbolicGroups} \cite{PV12} Any non-amenable, weakly amenable, biexact group is in $\mathcal{C}_\text{rss} $, e.g.~non-amenable hyperbolic groups.  
\item \cite{VV14} $\mathcal{C}_{\text{rss}} $ is closed under commensurabilty up to finite kernel.
\end{enumerate}
If $ \mathcal{F}$ denotes the collection of groups commensurable to the non-amenable free groups, $\mathcal{H} $ non-elementary hyperbolic groups and all non-amenable, non-trivial free product of exact groups, then we have the following elementary inclusions:
\begin{align*}
\mathcal{F}\subset \mathcal{H}\subset \mathcal{C}_\text{rss}.
\end{align*}
Hence all poly-hyperbolic groups no amenable factors in its composition series are each poly-$\mathcal{C}_\text{rss} $ groups.   
We note that In view of item (3) above, each collection of groups $\Quot_n(\mathcal{C}_\text{rss)} $ is closed under commensurability.  
\subsection{Mixed Surface Braid Groups}
Let $M $ be an orientable surface.  For every fixed positive integer $k$, $F_k(M)=\set{(x_1,\ldots, x_k)\in M^k : x_i\neq x_k\, \forall \, i\neq j} $ denotes the $k$-th configuration space of $M$.  The symmetric group $\mathcal{S}_k $ acts freely  on $F_k(M)$ by permutation of the coordinates.  The groups $PB_k(M):=\pi_1(F_k(M)) $ and $B_k(M):=\pi_1(F_k(M)/S_k) $ are the \emph{pure braid group of $k$ strands} and the \emph{braid group of $k$ strands} of the surface $M$, respectively.     When $M=\mathbb{D}^2 $. we recover the classical (pure/mixed)  braid groups, and in this case we simply denote  $B_k(\mathbb{D}^2),\, P_k(\mathbb{D}^2) $, or $ B_{k,j}(\mathbb{D}^2)$ by $B_k, P_k $ and $B_{j,k} $, respectively.  

Now, fix a pair of positive  integers  $k $ and $j$.  Regarded as a subgroup  of $S_{k+j} $, $S_k\times S_j $ has a natural free action on $F_{k+j}(M)$  by allowing $S_k $ to act on the first $k $ coordinates and $S_j $ on the remaining $j$ coordinates.  The group $B_{k,j}(M):=\pi_1(F_{k+j}(M)/ (S_k\times S_j)) $ is the \emph{mixed braid group} of $(k,j) $ strands.  Note that there is a canonical embedding of $B_{k,j}(M) $ into $B_{k,j}(M) $.  
Following as in \cite{CKP14}, $\widetilde{P}_k $, $\widetilde{B}_{k,j} $, and $\widetilde{B}_{j} $  will denote the central quotients of the pure, mixed, and standard braid groups.  

Fix $j$ distinct points $x_1,\ldots, x_j \in M $.  The map $F_{k+j}(M)/ (S_k\times S_j)  \to F_j(M) /S_j$ given by forgetting the first $k$ coordinates is a locally trivial fibration with fiber $F_k(M\setminus \set{x_1,\ldots ,x_j}$ called the \emph{Fadwell-Newarth fibration}.  The long exact sequence in homotopy of this fibration yields the following short exact sequence for the mixed braid groups:
\begin{align}\label{eq:MSBSequence}
1\rightarrow B_k(M\setminus\set{x_1,\ldots,x_j)})\to B_{k,j}(M)\rightarrow B_j(M) \to 1
\end{align}
provided that $M\neq S^2 $ or $\mathbb{R}P^2 $ \cite[Theorem 3, Corollary 2.2]{FN62}.  

Let $M =\Sigma_{0,b}$ be an orientable genus $g=0 $ surface with $b $ boundary components.  Letting $k,j\geq 3 $, Equation \eqref{eq:MSBSequence} becomes
\begin{align*}
1\rightarrow B_k(\Sigma_{0,b+j})\to B_{k,j}(M)\rightarrow B_j(M)\to 1,
\end{align*}
where $\Sigma_{0,b+j} $ is a surface with $b+j $ boundary components.  Passing to the central quotients, \cite[Theorem 3.5]{CKP14} implies that $\widetilde{B}_{k,n}(M) \in \Quot_{k+j-4}(\mathcal{F})$.

\subsection{Direct Product Decompositions for $\Quot(\mathcal{C}_\text{rss}) $}
Fixing a group $\Gamma $ in  $\Quot(\mathcal{C}_\text{rss}) $,  we begin by demonstrating that the existence of large commuting subgroups $\Sigma_1,\Sigma_2 $ inside the group $\Gamma $ implies the existence of a \emph{virtual} splitting of the group as a direct product. Essentially, we perturb the groups $\Sigma_i $ to commensurable  groups $\Lambda_i\leqslant\Gamma $ so that $\Lambda_1\times \Lambda_2 $ are $\Gamma $ are commensurable. This is achieved by examining the relationship between the analytic properties  and algebraic properties of groups in the class $\mathcal{C}\text{rss} $ and subsequently investigating how these properties translate to all groups in
  $\Quot(\mathcal{C}_\text{rss}) $.  

\begin{lemma}\label{lem:Quot1Prime}
Let  $p\in L(\Gamma) $ be a projection and $A, B \subset pL(\Gamma) p$ be two diffuse commuting subalgebras of $pL(\Gamma) p$ with $\Gamma\in \mathcal{C}_{\text{rss}} $.  Then $[pL(\Gamma) p : A\vee B]_{PP} =\infty$.  In particular if $ \Sigma_1, \Sigma_2<\Gamma$ are commuting subgroups such that $[\Gamma: \Sigma_1\Sigma_2]<\infty $, then either $\Sigma_1 $ or $\Sigma_2 $ is finite.  
\end{lemma}
We omit the proof, as it is the base step of the induction argument of the main theorem of \cite{CKP14}.

\begin{prop}\label{prop:ProductConverse}
$\Gamma\in \Quot_n({\mathcal{C}_\text{rss}}) $. Suppose there exist infinite groups such that $\Gamma$ is commensurable to $\Sigma_1\times \Sigma_2 $.  Then we may find $n_1,n_2>0 $ such that $n_1+n_2=n $ with $\Sigma_i\in \Quot_{n_i}(\mathcal{C}_\text{rss}) $. 
\end{prop}

\begin{proof}
 As $\Gamma\in \Quot_n(\mathcal{C}_{\text rss}) $, there exists a chain
\begin{align}\label{eq:PropChainQuotn}
\Gamma_n\overset{\pi_n}\twoheadrightarrow \Gamma_{n-1}\overset{{\pi_{n-1}}}\twoheadrightarrow \cdots \overset{\pi_2}\twoheadrightarrow \Gamma_1\overset{\pi_1}\twoheadrightarrow 1
\end{align}
with $\Gamma $ commensurable to $ \Gamma_n$.  As $\Quot_j(\mathcal{C}_{\text rss}) $ is closed under commensurability for all $ j$, after  passing to a finite index subgroup, we may assume $\Sigma_1\times \Sigma_2=\Gamma_n $. 

Let $\rho_n:\Gamma_n\to \Gamma_1 $ be given by the concatenation of the homomorphisms given  in \eqref{eq:PropChainQuotn}.  Note $\rho_n(\Sigma_1\times\set{e}),\rho_n(\set{e}\times \Sigma_2)$ are commuting groups generating $\Gamma_1 $.   Lemma \ref{lem:Quot1Prime} implies either $\rho_n(\Sigma_1\times\set{e}) $  or $\rho_n(\set{e}\times \Sigma_2) $ is finite.  Symmetry allows us to assume  $\ker(\pi_n|_{\Sigma_1})=\Sigma<\Sigma_1 $ is a finite index  normal subgroup of  $\Sigma_1$ with $\Sigma< \ker(\rho_n)=\Gamma_n^{(1)} $.  $\Sigma\times \Sigma_2 $ is a finite index subgroup of  $\Gamma_n $.  If we restrict $\rho_n $ to $\Sigma\times \Sigma_2$, a simple calculation yields $\Gamma_n^{(1)}=\Sigma\times \Lambda $ where  $\Lambda=\ker (\rho_n|_{\Sigma_2})$.  Thus we have
\begin{align}\label{eq:Sigma2Quot1}
\frac{\Sigma_2}{\Lambda}\cong \frac{\Sigma\times \Sigma_2}{\Sigma\times \Lambda}< \frac{\Gamma_n}{\Gamma_n^{(1)}}\in \mathcal{C}_{\text rss}
\end{align}
Since $\Gamma $ is commensurable to $\Sigma\times \Sigma_2 $, the inclusion in \eqref{eq:Sigma2Quot1} is a finite index inclusion of groups.  Thus $\Sigma_2/\Lambda\in \mathcal{C}_{\text rss} $. 

When  $n=2 $, we have $\Gamma_2^{(1)}\in\mathcal{C}_{\text rss} $  and $\Gamma_2^{(1)}=\Sigma\times \Lambda $   Lemma \ref{lem:Quot1Prime} implies $\Lambda $ is finite. Thus $ \Sigma$ is commensurable to $\Gamma_n^{(1)} $ yielding $\Sigma\in \Quot_1(\mathcal{C}_\text{rss}) $.  Since $\Lambda\triangleleft\Sigma_2 $ and $\Lambda $ is finite, passing to a finite index subgroup of both $\Lambda $ and $\Sigma_2 $, \eqref{eq:Sigma2Quot1} shows $\Sigma_2\in \mathcal{C}_{\text{rss}} $.  

Now if the result holds for all groups in $\Quot_{n-1}(\mathcal{C}_\text{rss}) $ up to some integer $n-1 $, we  have $\Gamma_n^{(1)}\in\Quot_{n-1}(\mathcal{C}_{\text rss}) $ with $\Gamma_n^{(1)}=\Sigma\times \Lambda $. If $\Lambda $ is a finite group, then repeating the argument when $n=2 $ proves the result for $n $.  If instead $\Lambda $ is an infinite group, then by the induction hypothesis, we may find $n_1,n_2>0 $ so that $ n_1+n_2=n-1$, $\Sigma\in \Quot_{n_1}(\mathcal{C}_\text{rss}) $, and $\Lambda\in \Quot_{n_2}(\mathcal{C}_\text{rss}) $. By \eqref{eq:Sigma2Quot1}, we then have $ \Sigma_2\in \Quot_{n_2+1}(\mathcal{C}_{\text rss})$. Furthermore, $\Sigma_1 $ is commensurable to $\Sigma$ which by Propositions \ref{prop:QuotProperties} ensures $\Sigma_1\in \Quot_{n_1}(\mathcal{C}_\text{rss}) $.
\end{proof}

\begin{cor}\label{cor:ProductConverseExact}
Let $\Gamma\in \Quot_n(\mathcal{C}_\text{rss}) $ and suppose $\Gamma $ is commensurable to $\Sigma_1\times \Sigma_2 $ with $\Sigma_1\in \Quot_j(\mathcal{C}) $ for some $ j\in \N$.  Then either $n=j $ and $\Sigma_2$ is finite, or $j<n $ and $\Sigma_2\in \Quot_{n-j}(\mathcal{C}_\text{rss}) $.
\end{cor}
\begin{proof}
By the minimality constraint in Definition \ref{defn:QuotnCCommensurable}, we naturally have $ j\leq n$.  If  $j=n $ and $\Sigma_2 $ were infinite, Proposition \ref{prop:ProductConverse} yields  $\Sigma_i\in\Quot_{n_i}(\mathcal{C}) $ for some $n>n_i\geq 1 $, once again contradicting minimality.  

Now suppose $1\leq j<n $.  If $n=2 $, by Proposition \ref{prop:ProductConverse}, we have the result. We momentarily define  $\Quot_0(\mathcal{C}_\text{rss}) $ as the collection of all finite groups. Proceeding as in the proof of the previous proposition, we have either
\begin{enumerate}
\item $\Gamma_n^{(1)} $ is commensurable to  $\Sigma_1\times \Lambda_2 $
\item $\Gamma_n^{(1)} $ is commensurable to $\Lambda_1\times \Sigma_2 $,
\end{enumerate}
where $\Lambda_i=\ker(\rho_n|_{\Sigma_i}) $.  In case (1),  $\Sigma_1\in \Quot_a(\mathcal{C}_\text{rss}) $ and $\Lambda_2\in \Quot_b(\mathcal{C}_\text{rss}) $, for $a,b\geq 0 $ with $a+b=n-1 $. Hence $\Sigma_2\in \Quot_{b+1}(\mathcal{C}_\text{rss}) $ By minimality $a\geq j $.  If $a>j $, this would imply $\Gamma\in \Quot_{j+b+1}(\mathcal{C}_\text{rss}) $ where $j+b+1<n $ once again contradicting minimality.  
In case (2), a similar argument will guarantee $\Lambda\in \Quot_{j-1}(\mathcal{C}_\text{rss}) $ and thus $\Sigma\in \Quot_{n-j}(\mathcal{C}_\text{rss}) $. 
\end{proof}
Further analysis of the claims above yields the following stronger statement.  We will omit the proof as it follows almost identically 
\begin{cor}\label{cor:QuotnCommutingGroups}
Let $\Gamma\in \Quot_n(\mathcal{C}_\text{rss}) $. Suppose we have a chain witnessing $\Gamma\in \Quot_n(\mathcal{C})_\text{rss} $
\begin{align*}
\Gamma_n\overset{\pi_n}\twoheadrightarrow \Gamma_{n-1}\overset{{\pi_{n-1}}}\twoheadrightarrow \cdots \overset{\pi_2}\twoheadrightarrow \Gamma_1\overset{\pi_1}\twoheadrightarrow 1
\end{align*}
with $[\Gamma_n:\Sigma_1\Sigma_2]<\infty $ for some infinite commuting groups $\Sigma_1,\Sigma_2< \Gamma_n $.  Then there exists $n_1, n_2>0 $ so that $n_1+n_2=n $ so that $\Sigma_i\in \Quot_{n_i}(\mathcal{C}_\text{rss}) $. Furthermore, if a priori we have $\Sigma_1\in \Quot_{j} $, then $\Sigma_2\in \Quot_{n-j}(\mathcal{C}_\text{rss}) $.
\end{cor}

\section{Proof of the Main Theorem}
Through their analysis of the generalized comultiplication map described within \cite{CIK13}, Chifan, Kida, and the second authors work in \cite{CKP14}   establish primeness results for the algebras coming from a  large natural class of groups.  These groups include most mapping class groups, central quotients of pure braid groups, Torelli groups, and Johnson kernels of punctured low-genus surfaces. We further Chifan's, Kida's  and the second author's analysis of the algebras coming from groups in $\Quot_{n}(\mathcal{C}_\text{rss}) $ by introducing to their work the techniques found within \cite{CdSS15}.  This combined approach will allow us to detect a direct product within the underlying group  $\Gamma$ whenever the algebra is non-prime.

To establish the main results, we set out the following notation taken \cite{CIK13}.  A group homomorphism $\rho:\Gamma\to \Lambda $ lifts $\rho $ to a $*-$homormorphism of von Neumann  algebras $\Delta:L(\Gamma)\to L(\Gamma)\bar{\otimes} L(\Lambda) $ by extending the map $u_\gamma\mapsto u_\gamma\otimes v_{\rho(\gamma) } $ where $\set{u_\gamma}_{\gamma\in \Gamma}, \set{v_\lambda}_{\lambda\in \Lambda}$ are the canonical unitaries of $L(\Gamma)$ and $L(\Lambda) $, respectively. When $\rho:\Gamma\to \Gamma$ is the identity, this is precisely the compultiplication map along $\Gamma $. For a group  $\Gamma\in\Quot_n(\mathcal{C}_{\text{rss}})$, we will consider the group homomorphism $\rho_n:\Gamma_n\to \Gamma_1 $ as defined in the previous section.   
If $P_1,\ldots, P_k\subset M  $ are subalgebras of $M$, then $P_1\vee \cdots\vee P_k \subset M $ is defined to be the smallest von Neumann  algebra in $M$ containing $P_1\cup \cdots \cup P_k$.  For every $j\in \set{1,\ldots, k}, $, we denote $\hat{P_j}:= P_1\vee\cdots \vee P_{j-1}\vee P_{j+1}\vee \cdots \vee P_k $.

\begin{lemma}\label{lem:TrivialNormalizer}
Let $ M$ be a type II$_1$ factor with $P_1,\ldots, P_k\subset M $ diffuse commuting II$_1 $ factors such that $P_1\vee \cdots P_k \subset M$ is a finite index inclusion of algebras.  Then there exists a projection $z\in M $ so that $\mathcal{Z}(Q_iz) \cong \C$,  	where $\mathcal{N}_{M}(P_i)'' =Q_i$.  
\end{lemma}
\begin{proof}
Letting  $P=P_1\vee \cdots\vee P_k $,  Theorem \ref{thrm:IndexofvnAlgebras} part (d) implies $P'\cap M $ is finite dimensional.  Notice $P\subset Q_i $ for every $i=1,\ldots, k $ since $P_i $ and $P_i'\cap M $ are both subalgebras of $ Q_i$.  Thus $\mathcal{Z}(Q_i) $ is finite dimensional since  $\mathcal{Z}(Q_i)\subset Q_i'\cap M\subset P'\cap M $.  

Now for any projection  $ z\in \mathcal{Z}(Q_i)$, we claim $Q_iz=\mathcal{N}_{zMz}(P_iz)'' $\footnote{This holds in much greater generality: the projection $z $ can be taken to be  any projection in $Q_i'\cap M $}.  It suffices to show $\mathcal{N}_{M}(P_i)z=\mathcal{N}_{zMz}(P_iz) $.  This follows clearly form the following facts: given any unitary $u\in\mathcal{N}_{M}(P_i) $, $(uz)^*uz=z=uz(uz)^* $; if $v\in \mathcal{U}(zMz) $ is a normalizing unitary of $P_iz $, then $v+(1-z)\in \mathcal{U}(M) $ is a normalizing unitary of $P_i $.

Since the algebras $\mathcal{Z}(Q_i) $ pairwise commute, we may take any  minimal projections $z_i\in \mathcal{Z}(Q_i)$ so that $z_iz_j=z_jz_1\neq 0 $.  Then
\begin{align*}
Az\subset zMz=N
\end{align*}
is a finite index inclusion of algebras with $\mathcal{Z}(\mathcal{N}_N(P_iz)'')=\mathcal{Z}(Q_iz)= \C z $.  
\end{proof}

\begin{prop}\label{prop:Intertwine1Level}
Let $\Gamma\in \Quot_n(\mathcal{C}_{\text{rss}}) $ with $p\in L(\Gamma) $ a non-zero projection. Suppose there exist $A, B\subset pL(\Gamma) p$ diffuse commuting subalgebras   so that $[pL(\Gamma) p: A\vee B]_{\text PP}<\infty$.  Then either
\begin{enumerate}
\item $A\prec_{L(\Gamma) } L(\Gamma_n^{(1)})  $, or
\item $B\prec_{L(\Gamma) } L(\Gamma_n^{(1)} )$.
\end{enumerate}
\end{prop}

\begin{proof}
First note $n\geq 2 $ by Lemma \ref{lem:Quot1Prime}.  Let $ P_0\subset A$ be an amenable subalgebra. Letting $M=L(\Gamma) $, we see $\Delta(P_0) $ and $\Delta(B) $ are diffuse commuting subalgebras of $M\bar\otimes L(\Gamma_1) $ with $\Delta(P_0) $ amenable.   Since $\Gamma_1\in \Quot_1(\mathcal{C}_{\text{rss}})=\mathcal{C}_{\text{rss}} $, we have either
\begin{enumerate}
\item $\Delta(P_0)\prec M\bar\otimes 1 $, or
\item $\mathcal{N}_{M\bar{\otimes} L(\Gamma_1)}(\Delta(P_0))''  $ is amenable relative to $M\bar\otimes 1 $ in $M\bar{\otimes} L(\Gamma_1) $.
\end{enumerate}
Assume (2) holds. Noting $\Delta(B)\subset \mathcal{N}_{M\bar{\otimes}L(\Gamma_1)}(\Delta(P_0))'' $  yields  $\Delta(B) $ is amenable relative to $M\bar\otimes 1 $ inside $M\bar\otimes L(\Gamma_1) $.  Applying the dichotomy property of $\mathcal{C}_{\text{rss}} $ either
\begin{enumerate}[resume]
\item $\Delta(B)\prec_{M\bar \otimes L(\Gamma_1)} M\bar\otimes 1$, or
\item $\mathcal{N}_{M\bar{\otimes} L(\Gamma_1)}(\Delta(B))'' $ is amenable relative to $M\bar\otimes 1 $ in $M\bar\otimes L(\Gamma_1) $.  
\end{enumerate}
To summarize we have either 
\begin{enumerate}[resume]
\item $\Delta(P_0)\prec_{M\bar{\otimes} L(\Gamma_1)} M\bar\otimes 1 $,
\item $\Delta (B) \prec_{M\bar{\otimes} L(\Gamma_1)} M\bar\otimes 1$, or
\item $\Delta(A\vee B)\subset \mathcal{N}_{M\bar{\otimes} L(\Gamma_1)}(\Delta(B))'' $ is amenable relative to $M\bar\otimes 1$ in  $M\bar\otimes L(\Gamma_1)$.
\end{enumerate}
We first show case (7) is impossible.  Since $A\vee B $ is  a finite index subalgebra of $M $, $M\prec^s A\vee B $ and hence $ M$ is amenable relative to $A\vee B $.  By \cite[Proposition 2.4]{OP07}, $\Delta(M) $ is amenable relative to $M\bar\otimes 1 $ in $M\bar\otimes L( \Gamma_1) $.  However, \cite[Proposition 3.5]{CIK13} would imply $\rho(\Gamma)=\Gamma_1 $ is amenable, a contradiction.  Thus we only have case (5) and (6).  Since $P_0 $ was an arbitrary amenable subalgebra of $A$, by \cite{BO08} we have either
\begin{enumerate}[resume]
\item $\Delta(A)\prec_{M\bar{\otimes} L(\Gamma_1)} M\bar\otimes 1 $, or
\item $\Delta (B) \prec_{M\bar{\otimes} L(\Gamma_1)} M\bar\otimes 1$.
\end{enumerate}
Noting $\ker(\rho_n)=L(\Gamma_n^{(1)}) $, the result follows from application of  \cite[Proposition 3.4]{CIK13}.  
\end{proof}

\begin{lemma}\label{lem:NonAmenableAlgebras}
Let $\Gamma\in \Quot_n(\mathcal{C}_{\text{rss}})$.  If $A,B\subset pL(\Gamma) p $ are diffuse commuting subalgebras  with $A $ amenable, then $[pL(\Gamma) p : A\vee B]_{PP}=\infty $.
\end{lemma}

\begin{proof}
We proceed by induction on $ n$.  When $n=1 $, this follows from Lemma \ref{lem:Quot1Prime}.  Now assume the statement holds for all groups in $\Quot_k(\mathcal{C}_{\text rss}) $ where $k\leq n-1 $ and take $\Gamma\in \Quot_n(\mathcal{C}_{\text rss}) $ with $A,B\subset pL(\Gamma) p $ as stated.  Then  $\Delta(A) $ is amenable and therefore amenable relative to $M\bar \otimes 1 $ in $M\bar\otimes L(\Gamma_1)$.  Since $\Gamma_1\in\mathcal{C}_{\text{rss}} $, either 
\begin{enumerate}
\item $\Delta(A)\prec_{M\bar \otimes L(\Gamma_1)} M\bar\otimes 1 $, or 
\item $\mathcal{N}_{M\otimes L(\Gamma_1)}({\Delta(A)})''$ is amenable relative to $ M\bar \otimes 1 $ in $M\bar\otimes L(\Gamma_1) $.  
\end{enumerate}
We first show (1) is impossible.  If (1) were to hold, \cite[Proposition 3.4]{CIK13} implies $A\prec L(\Gamma_n^{(1)} )$.  By \cite[Proposition 2.4]{CKP14}, there exists a $* $-isomorphism $\psi:p_1Ap_1\to P_1\subset qL(\Gamma_n^{(1)}) q $ such that $P_1 \vee ({P_1}'\cap qL(\Gamma_n^{(1)})q)\subset qL(\Gamma_n^{(1)})q $ is a finite index inclusion of algebras. Since $A $ is a diffuse amenable algebra, $P_1 $ is also diffuse amenable.  Since $\Gamma_n^{(1)} $ is non-amenable, $P_1'\cap qL(\Gamma_n^{(1)} )q$ is non-amenable. Supposing $P_1'\cap qL(\Gamma_n^{(1)}) q $ has an atomic corner, cutting by a minimal central projection $z $,  $P_1z \subset qzL(\Gamma_n^{(1)}) qz$ is a finite index inclusion of algebra.  Since $P_1z $ is amenable is an amenable corner of $L(\Gamma_n^{(1)}) $, this is would imply $\Gamma_n^{(1)} $ is an amenable group, a contradiction. If instead $P_1'\cap qL(\Gamma_n^{(1)})q $ were diffuse, this would contradict the induction hypothesis.  
 
Now if (2) holds, the assumption $[M:A\vee B]_{PP}<\infty $ implies $\Delta(A\vee B)\subset M\bar\otimes L(\Gamma_1) $ is also a finite index inclusion of algebras.  Since
$\Delta(A\vee B)\subset \mathcal{N}_{M\bar \otimes  L(\Gamma_1)}{(\Delta(A))}'' $, $\Delta(A\vee B) $ is amenable relative to $M\bar\otimes 
1 $ in $M\bar\otimes L(\Gamma_1)$.  By \cite[Proposition 2.4]{OP07}, $M\bar\otimes L(\Gamma_1 )$ is amenable relative to $M\bar\otimes 1 $.
 However, this is impossible as \cite[Proposition 3.5]{CIK13} would imply $\Gamma_1\in\mathcal{C}_{\text rss} $ is amenable. 
\end{proof}

To establish the main result, we show that the maximal number of commuting diffuse subalgebras is controlled by the Hirsh length. Note that we have an upper bound rather than equality. As an example of when we have a strict upper bound, central quotients of braid groups are poly-free groups which give rise to prime von Neumann algebras \cite[Theorem A]{CKP14}.  

\begin{lemma}\label{lem:MaxNumberCommutingAlg}
Let $\Gamma\in \Quot_n(\mathcal{C}_{\text{rss}}) $ and suppose $P_1,\ldots , P_k \subset qL(\Gamma) q$  are diffuse commuting II$_1$ factors.  If $P_1\vee \cdots \vee P_k\subset  qL(\Gamma) q$ generate a finite index subalgebra, then $ k\leq n$. 
\end{lemma}

\begin{proof}
  When $n=1 $, Lemma \ref{lem:Quot1Prime} proves the assertion. Now suppose the conclusion  holds for all groups in $\Quot_{m}(\mathcal{C}_{\text rss})$ up to $m= n-1$.  Now let $\Gamma\in \Quot_{n}(\mathcal{C}_{\text{rss}}) $ and assume to the contrary there are $k>n $ diffuse subalgebras $ P_1, \ldots, P_k\subset  qL(\Gamma) q$ generating a finite index subalgebra of $ qL(\Gamma) q$.  Without loss of generality, we may assume $k=n+1 $.  Then by Proposition \ref{prop:Intertwine1Level},  for every $j\in \set{1,\ldots, k} $, either 
$\hat{P_j}\prec_{L(\Gamma)} L(\Gamma_n^{(1)} )$ or $P_{j}\prec_{L(\Gamma)} L(\Gamma_n^{(1)} )$.

Now if $\hat{P_j}\prec L(\Gamma_n^{(1)}) $, by \cite[Proposition 2.4]{CKP14} there exists a $ *$-isomorphism $\psi:  p\hat{P_j}p\to A\subset  rL(\Gamma_n^{(1)})r$ so that $A\vee A'\cap rL(\Gamma_n^{(1)})r\subset rL(\Gamma_n^{(1)})r $ is a finite index inclusion of algebras. We  may assume $p=p_1\cdots p_k $, $p_i\in P_i $ for $i\neq j $.  Hence  $\psi(p \hat{P_j}p)= \psi(\bigvee_{i\neq j}  p_iP_i p) =\bigvee_{i\neq j}\psi(p_iP_i p) $.  Thus
\begin{align*}
\bigvee_{i\neq j}\psi (p_iP_i p) \vee (\psi( p \hat{P_j} p) ' \cap rL(\Gamma_n^{(1)})r )\subset rL(\Gamma_n^{(1)})r
\end{align*}
is a finite index inclusion of algebras. By Lemma \ref{lem:NonAmenableAlgebras}, the center $\mathcal{Z}(\bigvee_{i\neq j}\psi (p_iP_i p) \vee (\psi( p \hat{P_j} p) ' \cap rL(\Gamma_n^{(1)})r )) $ cannot be diffuse.  Thus, cutting by a minimal central projection we may assume 
\begin{align*}
\bigvee_{i\neq j}\psi (p_iP_i p) \vee (\psi( p \hat{P_j} p) ' \cap rL(\Gamma_n^{(1)})r )\subset rL(\Gamma_n^{(1)})r
\end{align*}
is a finite index inclusion of factors.  
 However, this would contradict the induction hypothesis as it would allow for at least $n $ commuting diffuse non-amenable subalgebras of $rL(\Gamma_n^{(1)}) r$. 
 
  If instead $P_j\prec_{L(\Gamma)} L(\Gamma_n^{(1)}) $ for all $j $, \cite[Lemma 2.5, Proposition 2.6]{Va10}, in conjunction with the factoriality of each $P_j $, imply $P_j\prec_{L(\Gamma)}^s L(\Gamma_n^{(1)}) $. 
  Proposition \ref{prop:StrongIntertwiningMultiple} would then give $L(\Gamma)\prec_{L(\Gamma)} L(\Gamma),  $ which implies $\Gamma_n^{(1)} $ is finite index in $\Gamma $
 once again leading to a contradiction.
\end{proof}
The following proposition is the key ingredient in decomposing a group as a product: if we may find an subgroup of $\Sigma< \Gamma \in \Quot_n(\mathcal{C}_\text{rss})$ then we may also find another subgroup commuting with $\Sigma $ so that $\Gamma $ is commensurable to the direct product $\Sigma\times \Omega $. The proof of this proposition closely follows the proof of \cite[Theorem 4.3]{CdSS15}.

\begin{prop}\label{prop:finiteIndexAlgebrastoGroups}
Let $\Gamma\in \Quot_n(\mathcal{C}_\text {rss}) $ be an icc group and denote by $L(\Gamma)=M $.  Suppose we have  subgroup $\Sigma<\Gamma $,  and a projection $p\in L(\Sigma)'\cap M $ so that $\Sigma\in \Quot_{j}(\mathcal{C}_\text{rss}) $ and
\begin{align*}
p[L(\Sigma)\vee (L(\Sigma)'\cap M)]p\subset pMp
\end{align*}
is a finite index inclusion of II$_1 $ factors.  
Then we may find commuting subgroups $\Sigma_1,\Sigma_2< \Gamma $ such that $ [\Sigma:\Sigma_1]<\infty $ and $[\Gamma:\Sigma_1\times\Sigma_2]<\infty $.  Furthermore, if $ \Sigma\in \Quot_j(\mathcal{C}_\text{rss})$, then $\Sigma_2\in\Quot_{n-j}(\mathcal{C}_\text{rss}) $. 
\end{prop}
\begin{proof}
Letting $\Sigma_2=\set{\gamma\in \Gamma : |\gamma|^\Sigma<\infty }$ and proceeding as in \cite[claim 4.7]{CdSS15}, we see $[\Gamma: \Sigma_2\Sigma]<\infty $. Now, the first half of  \cite[Claim 4.8]{CdSS15} demonstrates $\Sigma\cap \Sigma_2$ is amenable since it can be written as an increasing tower of amenable groups.  
 Let $\Gamma_1$ act trivially on $\C $, $\Gamma\cong \Gamma_n\to\cdots \to \Gamma_1 $ is a chain witnessing $\Gamma\in \Quot_n(\mathcal{C}_{\text rss}) $. Since $L(\Sigma\cap \Sigma_2) $ is amenable and $\Sigma $ normalizes $\Sigma_2\cap \Sigma $, the dichotomy of $\mathcal{C}_{\text rss} $ will imply $L(\Sigma_2\cap \Sigma)\prec \C 1$.  Thus, \cite[Proposition 2.6]{CdSS15} implies $\Sigma\cap \Sigma_2 $ is finite.  

Claims 4.9--4.12 in the proof of   \cite[Theorem 4.3]{CdSS15} provides the existence of  a subgroup $\Sigma_1\leqslant\Sigma $ satisfying $[\Sigma:\Sigma_1]<\infty $, $[\Gamma:\Sigma_1 \Sigma_2]<\infty $, and $[\Sigma_2,\Sigma_1]=\set{e} $.   Since $\Sigma\cap \Sigma_2 \geqslant \Sigma_1\cap \Sigma_2$, it follows $\Sigma_1\cap \Sigma_2 $  is finite as well. 
 Since $\Gamma $ is icc and $[\Gamma:\Sigma_1 \Sigma_2]<\infty $ then $ \Sigma_1\cap \Sigma_2=\set{e}$.  Thus $[\Gamma:\Sigma_1\times \Sigma_2]<\infty $.
 
  Now if we also had assumed $\Sigma\in\Quot_j(\mathcal{C}_\text{rss}) $, Corollary \ref{cor:QuotnCommutingGroups} yields $\Sigma_2\in \Quot_{n-j}(\mathcal{C}_\text{rss}) $.
\end{proof}

\begin{thrm}\label{thrm:main2Algebras}
Let $\Gamma\in \Quot_n(\mathcal{C}_\text{rss}) $ be an icc group and $q\in L(\Gamma)$ a projection.  Suppose  $P_1,\ldots, P_k\subset qL(\Gamma)q=M $ are diffuse commuting II$_1$ factors such that $[M: P_1\vee\cdots\vee P_k]<\infty $. Then  there exist icc groups $\Sigma_i\in \Quot_{n_i}(\mathcal{C}_\text{rss}) $, non-zero projections $p_i\in P_i $, finite index subfactors $D_i\subset p_iP_ip_i $, and a unitary $u\in M $ such that 
\begin{itemize}
\item $\Gamma$ is commensurable to  $ \Sigma_1\times \cdots\times \Sigma_k $,
\item $\sum_{i=1}^k n_i=n $,
\item $D_i \subset p_iu^*L(\Sigma_i)up_i$ is a finite index inclusion of II$_1$ factors. 
\end{itemize}

\end{thrm}

\begin{proof} 
As our theorem is taken up to commensurability, we will treat the case when $\Gamma=\Gamma_n$ where 
\begin{align*}
\Gamma_n\to \Gamma_{n-1}\to\cdots\to \Gamma_1\to 1
\end{align*}  
witnesses $\Gamma\in \Quot_n(\mathcal{C}_\text{rss}) $.  Furthermore, Lemma \ref{lem:TrivialNormalizer} implies we may assume $\mathcal{N}_{pMp}(P_i)'' $ is a factor.  
By Proposition \ref{prop:Intertwine1Level}, for every $i$ we have either $\hat{P_i}\prec_M L(\Gamma_n^{(1)})$ or $P_i\prec_M L(\Gamma_n^{(1)}) $. If we assume   $P_i\prec_M L(\Gamma_n^{(1)}) $, then $P_i\prec_M^s L(\Gamma_n^{(1)}) $ and hence $P_1\vee \cdots \vee P_k\prec L(\Gamma_n^{(1)}) $.  Since $P_1\vee\cdots\vee P_k $ is a finite index subalgebra of $L(\Gamma) $, then  $[\Gamma:\Gamma_n^{(1)}]<\infty $ contradicting that $\Gamma/\Gamma_n^{(1)}\in\mathcal{C}_\text{rss} $.  
Thus, 
there exists  $i $ such that $\hat{P_i}\prec_{M} L(\Gamma_n^{(1)}) $ but $P_i\not\prec_M L(\Gamma_n^{(1)}) $.  For simplicity, we consider the case $i=k $.   \cite[Proposition 2.4]{CKP14} give the existence of projections $p\in \hat{P_k}, q_1\in L(\Gamma_n^{(1)}) $, a partial isometry $v\in M $, and a $* $-isomorphism $\psi:p\hat{P_k}p\to B\subset q_1L(\Gamma_n^{(1)})q_1 $ such that
\renewcommand{\labelenumi}{(\alph{enumi})}
\begin{enumerate}
\item $B\vee (B'\cap q_1L(\Gamma_n^{(1)})q_1)\subset q_1L(\Gamma_n^{(1)}) q_1$ is a finite index inclusion of algebras,
\item $\psi(x)v=vx $ for all $x\in p\hat{A}_kp $.
\end{enumerate}
As in the proof of Lemma \ref{lem:MaxNumberCommutingAlg}, we may assume $p=p_1\cdots p_{k-1} $ where $p_i\in P_i $ are projections such that  $B=\psi(p\hat{P_k}p)=\psi(p_1P_ip)\vee \cdots \vee \psi(p_{k-1}P_{k-1}p)=B_1\vee \cdots \vee B_{k-1} $.  Thus we have
\begin{enumerate}[resume]
\item $\psi(p_iP_ip)=B_i $,
\item $\psi(x)v=vx $ for all $x\in p\hat{P_k}p $,
\item $B_1\vee \cdots \vee B_{k-1}\vee (B'\cap q_1L(\Gamma_n^{(1)})q_1) \subset q_1L(\Gamma_n^{(1)})q_1$ is a finite index inclusion of algebras.
\end{enumerate}

We first assume $n=2 $. In this case, Lemma \ref{lem:MaxNumberCommutingAlg} implies $k=2 $.
Since $\Gamma_n^{(1)}\in \Quot_1(\mathcal{C}_{\text rss}) $, Lemma \ref{lem:Quot1Prime} implies $\mathcal{Z}( B'\cap q_1L(\Gamma_n^{(1)})q_1)$ cannot have any diffuse part and therefore is completely atomic.  Thus multiplying $v $ by some minimal central projection $ q'\in B'\cap q_1L(\Gamma_n^{(1)})q_1$ so that $vq'\neq 0 $, we may assume $\psi(pP_1p)=B\subset q_1L(\Gamma_n^{(1)})q_1 $ is a finite index inclusion of factors.  Moreover, $\dim_\C(\mathcal{Z}(qL(\Gamma_n^{(1)}))q)\leq [qL(\Gamma_n^{(1)})q:B]_{PP}<\infty  $ since $B$ is a  II$_1 $ factor.  Thus, after multiplying again by a minimal central projection, we may assume $B\subset q_1L(\Gamma_n^{(1)})q_1 $ is a finite index inclusion of II$_1 $ factors.  
We claim there exists a projection $r\in L(\Gamma_n^{(1)})'\cap M $ such that.  
\begin{align}\label{claim:GroupANd CommutandFinite}
r[ L((\Gamma_n^{(1)})\vee L(\Gamma_n^{(1)})'\cap M) ]r\subset rMr
\end{align}
is a finite index inclusion of II$_1 $ factors.\\
To this end,   the downward basic construction \cite[Lemma 3.1.8]{Jo81} gives a projection $e\in q_1L(\Gamma_n^{1})q_1 $ and a subfactor $C\subset B\subset q_1L(\Gamma_n^{(1)})q_1=\generator{B,e} $ such that $[B:C]=[q_1L(\Gamma_n^{(1)})q_1:B] $, $Ce=eL(\Gamma_n^{(1)})e $ and $Ce\cong C $.  Then the restriction $\psi^{-1}:C\to D\subset pP_1p $ is a $* $-isomorphism such that $[pP_1p: D]<\infty $ with $\psi^{-1}(y)v^*=v^*y $ for all $y\in C $. Let $\theta: Ce\to C $ be the $*$-isomorphism given by $xe\mapsto x $ and denote by $v'=ev $.  If we suppose $v'=0 $, we would have $vv^*xe=xvv^*e =0$ for all $x\in B $.  As $ \generator{B,e}e=Be$, $vv^*t=0 $ for all $t\in \generator{B,e} $.  However, since $q $ is the central support of $e $ in $\generator{B,e} $, this would yield $vv^*=0 $.  Thus it follows that $\varphi=\psi^{-1}\circ \theta: eL(\Gamma_n^{(1)})e\to D $ is a $* $-isomorphism satisfying 
\begin{align}\label{eq:IntertwineBack1}
\varphi(y)w^*=w^* x\,\,  \text{for all } y\in eL(\Gamma_n^{(1)})e
\end{align}
where $w^*$ is the partial isometry from the polar decomposition of $v^*e=|v^*e|w^* $.  Note that $s=w^*w\in D'\cap  pMp  $ and $ww^*\leq e $.   Thus equation \eqref{eq:IntertwineBack1},  we obtain
\begin{align}\label{eq:Conjugatebyw}
w^*L(\Gamma_n^{(1)})w=\varphi(eL(\Gamma_n^{(1)})e)w^*w=Ds
\end{align}
\begin{align}
(w^*L(\Gamma_n^{(1)})w)'\cap sMs=(Ds)'\cap sMs.\label{eq:ConjugateCommutantw}
\end{align}
First note $P_2p\subset D'\cap pMp  $.  Since $D
\subset pP_1p $ is a finite index inclusion, so are the inclusions $D\vee P_2p\subset p(P_1\vee P_2)p\subset pMp $ and hence $D\vee P_2p\subset pMp $  is a finite index inclusion of algebras.  
 By the local index formula, we also have $Ds\vee s(D'\cap   M)s\subset sMs   $ is also a finite index inclusion of II$_1$ factors.  

  Let $r=ww^* $ and $u\in M$ a unitary with $w^*=ur $. Conjugating \eqref{eq:Conjugatebyw} and \eqref{eq:ConjugateCommutantw} by $u $ implies
$r[L(\Gamma_n^{(1)})\vee L(\Gamma_n^{(1)})'\cap M]r= L(\Gamma_n^{(1)})r\vee (L(\Gamma_n^{(1)})'\cap rMr)\subset rMr $ is a 
 finite index inclusion of  II$_1$ factors (after shrinking $r$ is necessary). By Proposition \ref{prop:finiteIndexAlgebrastoGroups}, there exists a finite index subgroup $\Sigma_1<\Gamma_n^{(1)} $  such that $\Sigma_i\in \Quot_{1}(\mathcal{C}_\text{rss}) $ with $[\Gamma:\Sigma_1\times\Sigma_2]<\infty $ and $rP_2r\subset ru^*L(\Sigma_2) ur$, where $\Sigma_2=V_\Gamma(\Gamma_n^{(1)}) $.  Since $ru^*L(\Gamma_n^{(1)})ur\subset rP_1r $ is a finite index inclusion of II$_1 $ factors, so is the inclusion $ru^*L(\Sigma_1)r\subset rP_1r $.  Performing the downward basic construction gives a subfactor $P_1f\subset ru^*L(\Sigma_1)r $ where $f\in P_1'\cap rMr $.  
 
Since $rP_2r\subset r(P_1'\cap M)r $ is a finite Pimnser-Popa index inclusion of algebras, so is the inclusion $rP_2r\subset ru^* L(\Omega)ur $.  Thus cutting once again by a minimal projection we have $r_2P_2r_r\subset r_2u^*L(\Sigma_2)ur_2 $ is a finite index inclusion of II$_1$ factors.

Now suppose the result holds for all icc groups $\Lambda\in \Quot_{n-1}({\mathcal{C}_\text{rss}}) $ for some  $n\in \N $. Take $\Gamma\in \Quot_n(\mathcal{C}_{rss}) $ an icc group.  Proceeding as in the case when $n=2 $, we may assume $\hat{P_k}\prec_{M}L(\Gamma_n^{(1)}) $.  More precisely, since the center of $P_1\vee \cdots \vee P_k $ is trivial, by \cite[Lamma 2.5, Proposition 2.6]{Va10} $\hat{P_k}\prec_{M}^s L(\Gamma_n^{(1)}) $.  
\cite[Proposition 2.4]{CKP14} give the existence of projections $p\in P_1, q\in L(\Gamma_n^{(1)}) $, a partial isometry $v\in M $ and a $* $-isomorphism $\psi:pP_1p\to B\subset q_1L(\Gamma_n^{(1)})q_1 $ such that
\begin{enumerate}[resume]
\item $B\vee (B'\cap q_1L(\Gamma_n^{(1)})q_1)\subset q_1L(\Gamma_n^{(1)})q_1 $ is a finite index inclusion of algebras,
\item $\psi(x)v=vx $ for all $x\in pP_1p $.
\end{enumerate}
If $B_k=B'\cap q_1L(\Gamma_n^{(1)})q_1 $  was not diffuse, we cutting by a minimal central projection to obtain $B\subset q_1L(\Gamma_n^{(1)})q_1 $ is finite Pimnser-Popa index inclusion of algebras. As before, $\mathcal{Z}(qL(\Gamma_n^{(1)})) $ is finite dimensional.  Thus, we cut by an appropriate minimal central projection to obtain a finite index inclusion of II$_1$ factors and proceed exactly as in the case when $n=2$.

 Now suppose $B$ and $ B_k =B'\cap q_1L(\Gamma_n^{(1)})q_1$ are both diffuse.  Then,  by cutting by a minimal central projection if necessary, we have $B\vee B_k\subset q_1L(\Gamma_n^{(1)})q_1 $ is a finite index inclusion of II$_1$ factors.  By the  induction hypothesis, there exists a unitary $w\in q_1L(\Gamma_n^{(1)})q_1 $,subgroups  $\Lambda1,\ldots, \Lambda_k $ of $\Gamma_n^{(1)} $ and projections $p_i\in B_i $, $q_i\in L(\Lambda_i)$  so that 
 \begin{itemize}
\item $w(p_1B_ip_1)w^*\subset q_1L(\Lambda_i)q_1 $ is a finite index inclusion of II$_1$ factors
\item $\Gamma_n^{(1)} $ is commensurable to $\Lambda_1\times \cdots\times \Lambda_k  $
\item $\Lambda\in \Quot_{m_1}(\mathcal{C}_\text{rss}) $ with $1\leq m_1<n-1 $.
\item $\sum m_i=n-1 $.
\end{itemize}
Letting $\Lambda=\Lambda_1\times \cdots\times \Lambda_k$ and proceeding as in the case when $n=2 $, there exists a projection $s $ such that
\begin{align*}
s[L(\Lambda)\vee L(\Lambda)'\cap M]s\subset sMs
\end{align*}
is a finite index inclusion of II$_1$ factors. 
  Applying Lemma \ref{prop:finiteIndexAlgebrastoGroups} and following the same procedure as in the case when $n=2$, we may find a finite index subgroup $\Lambda_1< \Lambda $ so that $r\hat{P_j}r\subset ru^*L(\Lambda_1)ur $  and $rP_kr\subset L(\Lambda_2) $ are finite index inclusions of II$_1 $ factors with $[\Gamma: \Lambda_1\times \Lambda_2]<\infty $.  Furthermore, we may assume $$r\hat{P_j}r = rP_1r\vee \cdots \vee rP_{k-1}r.$$
Letting $\Gamma_k=\Sigma_2 $ and once again applying the induction hypothesis, we may appropriately identify corners of $P_i $ with groups $\Gamma_i $ so that $\Gamma_i\in \Quot_{n_i}(\mathcal{C}_\text{rss}) $ with $n_1+\cdots+n_k=n$.

\end{proof}
The above result may be extended to amplifications of the algebra $L(\Gamma) $.  

\begin{cor}\label{cor:main2AlgebrasAmplfication}
Let $\Gamma\in \Quot_n(\mathcal{C}_\text{rss}) $ be an icc group and denote by $M=L(\Gamma)^t $.  If $P_1,\ldots, P_k $ are diffuse commuting II$_1$ factors such that $[M:P_1\vee \cdots P_k]<\infty $, then $k\leq n $  and there exist icc groups $\Sigma_1\in \Quot_{n_i}(\mathcal{C}_\text{rss}) $, non-zero projections $p_i\in P_i $, finite index subfactors $D_i\subset p_iP_ip_i $ and a unitary $u\in M $ such that
\begin{itemize}
\item $\Gamma$ is commensurable to  $ \Sigma_1\times \cdots\times \Sigma_k $,
\item $\sum_{i=1}^k n_i=n $,
\item $D_i \subset p_iu^*L(\Sigma_i)up_i$ is a finite index inclusion of II$_1$ factors. 
\end{itemize}
\end{cor}

\begin{proof}
Fix an integer $N>t>0 $.  Since  $M=L(\Gamma)^t $, there exists a projection $q\in M_N(\C)\bar\otimes L(\Gamma) $ so that $qM_N(\C)\bar\otimes L(\Gamma)  q=M $ and $\operatorname{Tr}(q)=t $, where $Tr:M_N(\C)\bar\otimes L(\Gamma) \to \C $ is the trace induced on  $ M_N(\C)\bar\otimes L(\Gamma)$ from the trace $ \tau$ on $L(\Gamma) $.  Now fix a projection $r\in M $ with $\tau(r) =1/N$ and $\tilde{r}\leq q $, where $\tilde{r} $ is the image of $r $ inside $M_n(\C)\bar\otimes L(\Gamma) $ in the obvious way.  Then $\tilde{r}u_0^*Mu_0\tilde{r} =sL(\Gamma)s $ for some non-zero projection $s\in L(\Gamma) $ and unitary $u_0\in M $.

Labeling $Q_i=ru_0^*P_iu_0r $ for each $i\in\set{1,\ldots, k} $, we see $\bigvee_{i=1}^k Q_i\subset sL(\Gamma)s$ is a finite index inclusion of algebras.   Theorem \ref{thrm:main2Algebras} now finished the proof.  
\end{proof}

\begin{cor}\label{cor:ProductExact}
Let  $\Gamma\in\Quot_n(\mathcal{C}_\text{rss}) $ be an icc group and suppose $L(\Gamma)^t\cong P_1\bar\otimes  P_2 $ for some diffuse von Neumann algebras $P_i $.  Then there exist groups $  \Gamma_1, \Gamma_2$, a unitary $u\in \mathcal{U}(L(\Gamma)) $, positive integers  $n_1,n_2 $, and  a scalar $s> 0 $  such that:
\begin{enumerate}
\item $\Gamma$ is commenarable to $\Gamma_1\times \Gamma_2 $ with $\Gamma_i\in \Quot_{n_i}(\mathcal{C}_\text{rss}) $ and $n_1+n_2=n $; and
\item $P_{1}^s=u L(\Gamma_1)u^* $ and $P_{2}^{t/s}=u L(\Gamma_2)u^* $.
\end{enumerate}
\end{cor}
Since the proof closely follow the proof of \cite[Theorem 4.14]{CdSS15}, we include only the relevant details.
\begin{proof}
By Theorem \ref{cor:main2AlgebrasAmplfication}, there exists groups $\Gamma_i \in \Quot_{n_{i}}(\mathcal{C}_\text{rss}) $, a unitary $ u\in L(\Gamma)^t$, and  projections $p_i\in P_i  $, and finite subfactors  $D_i\subset p_iP_ip_i$ so that
\begin{enumerate}
\item $\Gamma $ is commensurable to $\Sigma_1\times \Sigma_2 $,
\item $n_1+n_2=n $, and
\item  $D_i\subset p_iu^*L(\Sigma_i)up_i$.
\end{enumerate}
Thus, all that remains to show is the existence of a scalar $s $ so that $P_1^s=uL(\Gamma_1)u^* $ and $P_2^{t/s}=uL(\Gamma_2)u^* $.  Since $\Gamma $ is commensurable to $\Sigma_1\times \Sigma_2 $, there exist commuting non-amenable subgroups $ \Omega_1,\Omega_2\leqslant\Gamma$ such that $[\Gamma:\Omega_1\Omega_2]<\infty $.   Thus, we are now in position to follow the proof of \cite[Theorem 4.14]{CdSS15} exactly.
\end{proof}
Proceeding by way of induction, we obtain the following generalization of  Corollary \ref{cor:ProductExact}
\begin{cor}\label{cor:ProuctExactMultiple}
Let  $\Gamma\in\Quot_n(\mathcal{C}_\text{rss}) $ be an icc group and suppose $L(\Gamma)\cong P_1\bar\otimes\cdots\bar\otimes   P_k $ for some diffuse von Neumann algebras $P_1.\ldots, P_k $.  Then there exist subgroups $  \Gamma_1, ,\ldots \Gamma_k\leqslant\Gamma$, a unitary $u\in \mathcal{U}(L(\Gamma)) $, positive integers  $n_1,\ldots, n_k $, and  scalars $t_1,\ldots, t_k> 0 $ with $t_1 t_2\cdots t_k=t $  such that:
\begin{enumerate}
\item $\Gamma=\Gamma_1\times \cdots \times\Gamma_k $ with $\Gamma_i\in \Quot_{n_i}(\mathcal{C}_\text{rss}) $ and $\sum_{i=1}^kn_i=n $; and
\item $P_{i}^{t_i}=u L(\Gamma_i)u^* $ for every $i\in \set{1,\ldots, k} $.
\end{enumerate}
\end{cor}

\begin{proof}
When $k=2 $,  this  follows from Corollary \ref{cor:ProductExact}.  Now assume the conclusion of the corollary holds for all tensor decompositions with at most $k-1 $ factors.  Writing $L(\Gamma)^t=Q_1\bar\otimes Q_2 $ where $Q_1=P_1\bar\otimes\cdots\bar\otimes P_{k-1} $ and $Q_2=P_k $, Corollary \ref{cor:ProductExact} implies the existence of subgroups $\Gamma_1, \Gamma_2\leqslant \Gamma$, a unitary $ u\in \mathcal{U}(L(\Gamma))$, and a scalar $s_1 $ so that $Q_1^{s}=uL(\Gamma_1)u^* $ and $Q_2^{t/s}=uL(\Gamma_2)u^* $.  Applying the induction hypothesis to $Q_1$ will give the conclusion of the corollary. 
\end{proof}

\begin{cor}
Suppose $ \Gamma_1,\ldots, \Gamma_m $ and $\Lambda_1,\ldots, \Lambda_m $ are icc groups such that $\Gamma_i\in \Quot_{n_i}(\mathcal{C}_\text{rss}) $ and $\Lambda_{i}\in \Quot_{m_j}(\mathcal{C}_\text{rss}) $.  If $L(\Lambda_i) $ and $L(\Gamma_j) $ are prime II$_1$ factors so that $L(\Gamma_1\times \cdots\times \Gamma_m)\cong L(\Lambda_1\times \cdots\times \Lambda_n) $, then  $n=m $ and we have $L(\Gamma_i)\cong L(\Lambda_i)$ for every $i\in \set{1,\ldots, m} $, up to		 permutation and amplification.
\end{cor}

\begin{proof}
The statement that $n=m $ clearly follows from Lemma \ref{lem:MaxNumberCommutingAlg} and Theorem \ref{thrm:main2Algebras}.  
Letting  $P_i=L(\Lambda_i) $ for every $i\in\set{1,\ldots m} $, Corollary \ref{cor:ProuctExactMultiple} immediately implies the conclusion.  
\end{proof}

We now provide explicit examples to which  we  apply our results.
\begin{cor}\label{cor:MixedBraid}
 If $ k,j\geq 3$  and $\widetilde{B}_{j,k}  \leq \widetilde{B}_{k+j}$ is a mixed braid subgroup, then  $pL(\widetilde{B}_{j,k}) p$ is prime for every non-zero projection $p\in \mathcal{P}(L(\widetilde{B}_{k,j})) $.     
\end{cor}
\begin{proof}
Suppose to the contrary that there exists a projection $p\in \mathcal{P}(L(\widetilde{B}_{k,j})) $ so that $pL(\widetilde{B}_{k,j})p $ is non-prime.  The  by Theorem \ref{thrm:main2Algebras},  $\widetilde{B}_{k,j} $ is commensurable to a product of groups.  However, since $B_{k,j} $ is a finite index subgroup of $\widetilde{B}_{k+j}$ of index $k!j! $, this would impy that $ \widetilde{B}_{k+j} $ is commensurable to a product of infinite groups, which contradicts \cite[Theorem A]{CKP14}.    
\end{proof}

Fix an integer $n\geq 3$ and choose a collection of at least two positive integers  $k_1,\ldots, k_m $  such that $k_i\neq 2$ for all $i\in\set{1,\ldots, m} $ with $k_1+\cdots + k_m=n $.  Then there exists a canonical epimorphism $\pi: B_n\to S_{k_1}\times \cdots \times S_{k_m} $. If we denote by $B_{k_1,\ldots, k_m}:=\ker \pi $, then this generalizes the mixed braid groups.  Note that since $[B:B_{k_1,\ldots, k_m}]=\prod_{i=1}^m k_i!  $, the proof of Corollary \ref{cor:MixedBraid} can be modified to verify the following:
\begin{cor}
Let $\widetilde{B}_{k_1,\ldots, k_m} $ be as above.  Then for any projection $p\in \mathcal{P}(L(\widetilde{B}_{k_1,\ldots, k_m}) )$, $pL(\widetilde{B}_{k_1,\ldots, k_m})p $ is prime.  
\end{cor} 

\section{Closing remarks}
Note the analysis  involved in the proofs of the theorems in Section 5 continue to hold if we forgo the factoriality assumption and instead   assume $ L(\Gamma) $ has a  finite dimensional center, e.g.   $\Gamma $ is a finite-by-icc group.  By carrying out the same procedure we have the following generalization of the main theorem:
\begin{thrm}
Suppose $P_1\vee\cdots \vee  P_k\subset L(\Gamma) $ where $\Gamma\in \Quot_n(\mathcal{C}_\text{rss}) $.  Then there exist commuting groups $\Sigma_i\in \Quot_{n_i}(\mathcal{C}_\text{rss})  $ so that $\Gamma $ is commensurable to the product $\Sigma_1\cdots \Sigma_k $.  
\end{thrm}
While our methods provide a criterion whereby one determines if a group von Neumann algebra which arises from a finite-step extension by $\mathcal{C}_\text{rss} $ is prime, we have not addressed the  question of \cite{CKP14}: is the $NC_1 $ condition sufficient to conclude that the resulting group von Neumann algebra is prime?

\end{document}